\documentclass{amsart}
\usepackage{amssymb,amsfonts,pb-diagram,bbm}
\theoremstyle{plain}
\newtheorem{thm}{Theorem}
\newtheorem{lem}[thm]{Lemma}
\newtheorem{cor}[thm]{Corollary}
\newtheorem{prop}[thm]{Proposition}
\theoremstyle{remark}
\newtheorem*{exmp}{Example}
\newcommand\moy{\mathcal O_{Y}}
\newcommand\moty{\mathcal O_{\tilde Y}}
\newcommand\mo{\mathcal O}
\newcommand\lra{\longrightarrow}
\newcommand\tx{\tilde X}
\newcommand\ty{\tilde Y}
\newcommand\omtx[1][1]{\Omega^{#1}_{\tilde X}}
\newcommand\omx[1][1]{\Omega^{#1}_{X}}
\newcommand\omty[1][1]{\Omega^{#1}_{\tilde Y}}
\newcommand\omy[1][1]{\Omega^{#1}_{Y}}
\newcommand\CC{\mathbb C}
\newcommand\PP{\mathbb P}
\newcommand\kk{\mathbbm k}
\newcommand\im{\operatorname{Im}}
\newcommand\JF{\operatorname{Jac}}
\begin{document}
\title{Defect formula for nodal complete intersection threefolds}
\author{S\l awomir Cynk}
\address{Institute of Mathematics,
Jagiellonian University,
ul. {\L}ojasiewicza 6,
30-348 Krak\'ow,
Poland}
\email{slawomir.cynk@uj.edu.pl}
\address{{Institute of Mathematics of the
Polish Academy of Sciences, ul. \'Sniadeckich 8, 00-956 Warszawa,
Poland}}
\thanks{Research partially supported by the National Science Center grant no. 2014/13/B/ST1/00133}
\thanks{This work was partially supported by the grant 346300 for IMPAN from the Simons Foundation and the matching 2015-2019 Polish MNiSW fund}

\begin{abstract}
We generalize Werner's defect formula for nodal hypersurfaces in
$\mathbb P^{4}$ to the case of a nodal complete intersection
threefold.  
 \end{abstract}
\maketitle

\vspace*{-1ex}
\section{Introduction}
\label{sec:intro}

The main goal of this paper is to give a formula for Hodge numbers of
a nodal complete intersection threefold satisfying certain
non--degeneracy condition. 
Hodge numbers of a transversal complete intersection in a projective
space can be computed from
the generating function of $\chi_{y}$--genus  \cite[Thm. 22.1.1,
Thm. 22.1.2]{Hirz}. In the special case of a threedimensional complete
intersection $X$ of hypersurfaces of degrees $(d_{1},d_{2},\dots, d_{r})$
in $\PP^{r+3}$ we can use the Hirzebruch--Riemann--Roch theorem for
the vector bundle $\omx$ and the Lefschetz hyperplane
theorem to compute
\[h^{1,2}(X)=\tfrac1{24}c_{1}c_{2}-\tfrac12c_{3}+1\]
and then
\begin{eqnarray*}
  h^{1,2}(X)=\left( {\tfrac {11}{24}}\,{{\it \sigma_1}}^{3}-{\tfrac {5\left( r+4 \right) }{12}}\,
    {{\it \sigma_1}}^{2}+ \left( \tfrac{\left( r+4 \right)  \left( 9\,r+25 \right)}{48}\,
       -{\tfrac {11}{12}}
      \,{\it \sigma_2} \right)  {\it \sigma_1}+\right.\rule{2cm}{0cm} \\
\left.{\tfrac {5\left( r+4 \right) }{12}}\,
    {\it \sigma_2}+\tfrac12\,{\it \sigma_{3}}-\tfrac{\left(
      3\,r+4 \right)  \left( r+4 \right)  \left( r+3 \right)}{48}\,   
 \right) {\it \sigma_{r}}+1
\end{eqnarray*}
where $\sigma_{i}$ is the $i$--th elementary symmetric function
evaluated at $(d_{1},d_{2},\dots, d_{r})$.
If $X=\{F=0\}$ is a degree $d$ hypersurface in $\PP^{4}$ there is moreover
isomorphism
\[H^{2,1}(X)\cong (\kk[X_{0},\dots,X_{4}]/\JF(F))_{2d-5}\]
of the Hodge group with degree $2d-5$ component of the 
Jacobian algebra of $X$ (an explicit isomorphism is described in
\cite{pet-st}). 

First formulae for the Hodge numbers of singular threefolds were given
by Clemens \cite{Clemens} for double coverings of $\PP^{3}$ branched along a nodal
double surface and then by Werner \cite{Werner} for nodal hypersurfaces in
$\PP^{4}$.  Clemens' and Werner's formulae relate the Hodge numbers of
a resolution of a nodal double solid and a nodal hypersurface to
the defect of certain linear system. These results were reproved with
algebraic methods (characteristic free) and generalized to the case of
hypersurfaces with A-D-E singularities satisfying certain
vanishings. The proofs follow the line of \cite{pet-st}, vanishing of
a certain cohomology group breaks--up the long cohomology sequence.

Our goal is to generalize Werner's formula to the case of a nodal
complete intersection in projective space, in this case the considered
exact sequence does not break, instead of vanishing we explicitly
describe the image of one of the maps in the sequence. Three
dimensional node admit two types of a special resolution. The first
one is the blow--up of the singular locus and is called big
resolution. Small resolution replace singular point with a line, in
general small resolution need not be projective. In our proofs we
consider the big resolution, but the Hodge numbers of any small one
follows easily.

Nodal threefolds play important role in several branches of algebraic
geometry, first examples of Calabi--Yau threefolds with small absolute
value of the Euler characteristic were constructed as
small resolutions of nodal hypersurfaces and complete intersection
(cf. \cite{vGW, Hirz, WvG, schoen, Meyer}). A $\mathbb Q$--factorial
nodal quartic 3--folds and nodal double sextic are non--rational which
raised the question of minimal number of nodes on non--$\mathbb
Q$--factorial nodal threefold of given type
(cf. \cite{Chelt, CR2, Ko, Kl, cheltsov2}). Special properties of
small resolutions of nodal threefolds were used to constructed
examples of Calabi--Yau spaces in positive characteristic
non--liftable to characteristic zero. Contraction of a class of lines
on a Calabi--Yau threefold to nodes followed by a smoothing of the
nodal threefolds is the so--called conifold transition which can connect
different families of Calabi--Yau threefolds (\cite{reid}). 
\section{Preliminaries }

Let $X=H_{1}\cap\dots\cap H_{r}\subset\PP^{r+3}$  be a nodal complete
intersection in $\PP^{r+3}$ of smooth hypersurfaces of dimensions
$d_{1},\dots,d_{r}$, denote $d:=d_{1}+\dots+d_{r}$. Assume moreover
that the intersections  
$Y=H_{1}\cap\dots\cap H_{r-1}$ is
smooth. 

We have the
following Bott--type vanishings 
 \[H^{i}(\Omega_{Y}^{j}(kX))=0, \quad \text{ for }i+j>4,\, k>0.\]

Let $\Sigma:=\operatorname{Sing}X$ be the singular locus of $X$,
$\mu=\#\Sigma$ -- the number of nodes of $X$ and let $\sigma:\ty\lra Y$ be the
blow--up of $Y$ at the singular locus of $X$. Denote by $\tx$ the
strict transform of $X$, let $E:=\sigma^{-1}(\Sigma)$ be the
exceptional divisor of $\sigma$. Then $\tx$ is non--singular and $E$
is a disjoint union of projective 3--spaces. 


\begin{prop}
  \begin{eqnarray*}
    &&H^{0}(\omty[4](\tx))\cong H^{0}(\omy[4](X)),\\
    &&H^{i}(\omty[4](\tx))=0,\qquad\text { for }i>0,\rule{5cm}{0cm}\\
    &&H^{i}(\omty[4](2\tx))\cong H^{i}(\omy[4](2X)\otimes\mathcal
       J_{\Sigma}),\qquad\text { for }i\ge0. 
  \end{eqnarray*}
\end{prop}
\begin{proof}
  We have
  $\omty[4](\tx)\cong\sigma^{*}\omy[4](X)\otimes\moty(E)$, first two
  assertions follows now from $\sigma_{*}\moty(E)\cong\moy$,
  $R^{i}\sigma_{*}\moty(E)=0$, projection formula and (degenerate
  case) of Leray spectral sequence. Applying the direct image functor
  to the exact sequence $0\lra \moty(-E)\lra \moty \lra \mo_{E} \lra
  0$ we get $\sigma_{*}\moty(-E)\cong \mathcal J_{\Sigma}$ and
  $R^{i}\sigma_{*}\moty(-E)=0$, the last assertion follows now from 
  $\omty[4](2\tx)\cong\sigma^{*}\omy[4](2X)\otimes\moty(-E)$. 
\end{proof}

\begin{cor}
  We have the following exact sequence
\[H^{0}\omy[4](2X)\lra H^{0}(\omy[4](2X)\otimes\mo_{\Sigma}) \lra
  H^{1}\omtx[3](\tx) \lra 0
\]
\end{cor}
\begin{proof}
  By adjunction formula $\omtx[3](\tx)\cong
  \omty[4](2\tx)\otimes\mo_{\tx}$, assertion follows now from the
  previous proposition and the long exact sequence associated to 
\[0 \lra \omty[4](\tx) \lra \omty[4](2\tx) \lra
  \omty[4](2\tx)\otimes\mo_{\tx} \lra 0.\]
\end{proof}
\begin{prop}\label{prop2}
  \begin{eqnarray*}
    &&H^{i}(\omty[3](\tx))\cong H^{i}(\omy[3](X)), \quad  i\ge0,\\
    &&H^{i}(\omty[3])=0, \quad  i\le2,
  \end{eqnarray*}
\end{prop}
\begin{proof}
  By direct computations in local coordinates we verify 
  \[\sigma^{*}\omy[3]\cong\omty[3](\log E)(-3E)\]
  and so 
  \[\sigma^{*}(\omy[3](X))\cong\omty[3](\log
    E)(-3E)\otimes\sigma^{*}\moy(X).\]

  Tensoring the exact sequence
  \[0\lra\omty[3](\log E)(-E) \lra \omty[3] \lra \Omega^{3}_{E} \lra
    0\]
  with $\moty(\tx)\cong \moty(-2E) \otimes \sigma^{*}(\moy(X))$
  we get
  \[0 \lra \sigma^{*}(\omy[3](X)) \lra \omty[3](\tx) \lra \mo_{E}(-2)
    \lra 0.\]
  Now, using the direct image operator and projection formula we get 
  \[\sigma_{*}\omty[3](\tx)\cong\omy[3](X)\qquad \text{and} \qquad
    R^{i}\sigma_{*}\omty[3](\tx)=0,\]
  the assertion follows from the Leray spectral sequence. Second
  assertion follows in a similar manner from the exact sequence
  \[0 \lra \sigma^{*}(\omy[3])\otimes \moty(2E) \lra \omty[3] \lra \Omega^{3}_{E}
    \lra 0\]
 and the Lefschetz hyperplane theorem $H^{i}(\omy[3])=0$.
\end{proof}

\begin{lem}
The following sequence is exact
  \begin{eqnarray*}
    &&0\lra H^{1}\omy[3] \lra H^{1}\omty[3](\log \tx) \lra
       H^{1}\Omega^{2}_{\tx} \lra 0\\
    &&
  \end{eqnarray*}
\end{lem}

\begin{proof}
  We have  $H^{0}(\omtx[2])=H^{2}(\mo_{\tx})=0$ (\cite[Prop.~3]{CR})
  and $H^{2}(\omty[3])=0$ (Prop.~\ref{prop2}), now the assertion  
  follows by the long cohomology exact sequence derived from
  \[0\lra\omty[3] \lra\omty[3](\log\tx) \lra \omtx[2]\lra 0\]
\end{proof}
\begin{lem}
  The following sequence is exact
  \begin{eqnarray*}
0\lra H^{0}(\omy[3](X))\lra H^{0}(\omtx[3](\tx)) \lra
  H^{1}(\omty[3](\log\tx)) \lra\rule{16mm}{0cm}\\\lra H^{1}(\omy[3](X))\lra
  H^{1}(\omtx[3](\tx))     
  \end{eqnarray*}
\end{lem}
\begin{proof}
  Follows from the short exact sequence
  \[0\lra\omty[3](\log\tx) \lra \omy[3](X)\lra
  \omtx[3](\tx)\lra0\]
and previous lemmata.
\end{proof}

\section{Main result}

Now, we shall formulate and prove our main result

\begin{thm}\label{t:main}
Let $F_{1},\dots,F_{r}\in S:=\kk[X_{0},\dots,X_{r+3}]$ be homogeneous
polynomials in $r+4$ variables such that
\begin{itemize}
\item varieties $V(F_{1},\dots,F_{i})$
  are smooth for $i=1,\dots,r-1$,
\item variety $X:=V(F_{1},\dots,F_{r})$ is a threefold with ordinary
  double points as the only singularities.
\end{itemize}

Denote by $\Sigma:=\operatorname{Sing}(X)$ the set of singular points
of $X$, $\mu:=\#\Sigma$ number of its elements and
$d:=d_{1}+\dots+d_{r}$. 
Let $V$ be a linear combination of rows of the matrix
$\bigwedge^{r-1}\operatorname{Jac}  (F_{1},\dots,F_{r})$ which does
not vanish at any point of $\Sigma$ and let $I$ be the ideal generated
by entries of $V$.

Then 
\[h^{1,1}(\hat X)=1+\delta,\quad h^{1,2}(\hat
  X)=h^{1,2}(X_{\operatorname{smooth}})-\mu+\delta\] where  
\[\delta :=\mu-(\dim_{\kk}I^{2d-2r-3}-\dim_{\kk}(I\cap\mathcal J_{\Sigma})^{2d-2r-3})\] 
is the defect of the ideal $I$ at the singular locus of $X$.
\end{thm}

\begin{lem}
  There exists an epimorphism 
  \[\bigoplus_{i=1}^{r-1}S^{d+d_{i}-r-4} \lra
    H^{1}\omy[3](X).\] 
\end{lem}
\begin{proof}
  Let $Z$ be a complete intersection of $r-2$ hypersurfaces $H_{i}$.
  Using Bertini theorem we can assume without lost of generality that
  $Z:=H_{1}\cap\dots\cap H_{r-2}$ is a smooth fivefold.
  By similar arguments as before we easily get 
  exact sequences 
  \begin{eqnarray*}
   H^{1}\Omega^{4}_{Z}(\log Y)(X)\lra H^{1}\omy[3](X)\lra 0&&\\
   H^{0}\omy[4](X)\otimes\mo_{Z}(Y)\lra H^{1}\Omega^{4}_{Z}(\log Y)(X) \lra
  H^{1}\Omega^{4}_{Z}(X+Y)\lra0&&
  \end{eqnarray*}
  By adjunction and the Bott vanishing we get recursively that
  $H^{0}(\omy[4](X)\otimes\mo_{Z}(Y))$ is an image of
  $S^{d+d_{r-1}-r-4}$. Now, the lemma follows by 
  induction.
\end{proof}
Consider the following commutative diagram
\begin{equation*}
  \begin{diagram}\dgARROWLENGTH=1.2em
\node[2]{\bigoplus_{i=1}^{d_{r}-1}S^{d+d_{i}-r-4}}
\arrow{s,r}{\beta}\arrow{e,t}{\alpha}
\node{H^{1}\left(\Omega^{3}_{Y}(X) \right)}
\arrow{s,r}{\phi}\arrow{e} \node{0}\\
    \node{H^{0}(\Omega^{4}_{Y}(2X))} \arrow {e,t}{\delta}
    \node{H^{0}(\Omega^{4}_{Y}(2X)\otimes \mo_{\Sigma})} \arrow{e,t}{\gamma}
    \node{H^{1}(\Omega^{3}_{\tx}(\tx))} \arrow{e}\node0\\
\node{S^{d+d_{r}-r-4}}\arrow{n,l}{\eta}\arrow{ne,l}{\theta}
\node{\kk^{\mu}}\arrow{n,r}{\cong}
\node{H^{1}(\Omega^{4}_{Y}(2X)\otimes \mathcal J_{\Sigma})}
\arrow{n,r}{\cong}
  \end{diagram}
\end{equation*}
All the maps except $\beta$ are determined by the proofs we presented,
on the other hand the identification $H^{0}(\Omega^{4}_{Y}(2X)\otimes
\mo_{\Sigma}) \cong \kk^{\mu}$ is not given explicitly. 

Denote by $\Omega$ the form
$\Omega:=\sum_{i=0}^{r+3}X_{i}dX_{0}\wedge\dots\wedge
\widehat{dX_{i}}\wedge \dots\wedge dX_{r+3}$. 
The map $\theta$ to a function $A$ associates Poincare residue of the
form $\frac A{F_{1}\dots F_{r-1}F_{r}^{2}}\Omega$ with respect to
$\frac{dF_{1}}{F_{1}},
\frac{dF_{2}}{F_{2}},\dots,\frac{dF_{r-1}}{F_{r-1}}$ evaluated at
points of $\Sigma$.
When we want to identify values of $\theta$ with vectors we have to
evaluate coefficients of resulting form, which is the same as evaluate
quotients of $A$ by $(r-1)\times(r-1)$ minors of the jacobian matrix
of $F_{1},\dots,F_{r-1}$. 

At each point of $\Sigma$ the jacobian matrix $\JF(F)$ has rank $r-1$, so
the matrix $\bigwedge^{r-1}\JF(F)$ of  $(r-1)\times(r-1)$ minors has rank 1.
By our assumption all the rows of this matrix are non--zero, so at
every point of $\Sigma$ some columns are zero the other columns have
are proportional and have only non--zero entries. It may happen however
that each column vanish at some point of $\Sigma$. In order to
circumvent this problem we take a random linear combination of columns
$(V_{1},\dots,V_{r})$ which does not vanish at any point. 

Composing with $\alpha,\gamma,\phi$, we see
that $\beta$ can be identified in the same manner as $\theta$ through
remaining  $(r-1)\times(r-1)$ minors of the jacobian matrix $\JF(F)$
of $F_{1},\dots,F_{r}$, main difference is that from $S^{d+d_{i}-r-4}$
we pass through $H^{0}(\omy[4](X)\otimes\mo_{Z}(Y))$ instead of
$H^{0}(\omy[4](2X))$ which means that we have to multiply by
$F_{r}/F_{i}$. Evaluating at a singular point we have to pass to the
limit equal $V_{i}/V_{r}$.
Finally, denoting $\Sigma:=\{P_{1},\dots,P_{\mu}\}$ the value of
$\beta$ at $A_{i}\in S^{d+d_{i}-r-4}$ is
$\frac{V_{i}(P)A_{i}(P)}{V_{r}(P)^{2}}$. 
Denote the ideals $I=(V_{1},\dots,V_{r})$, $J=I\cap\mathcal
J_{\Sigma}$ and by $I^{k}$ (resp. $J^{k}$) vector space of degree $k$
forms in $I$ resp. $J$. 
We have proved the following proposition
\begin{prop}
\[\dim(\im(\delta)+\im(\beta))=\dim I^{2d-2r-3}-\dim J^{2d-2r-3}.\]  
\end{prop}

\begin{proof}[\mbox{Proof of Thm.~\ref{t:main}}]
By simple linear algebra we get 
\[h^{1}(\omtx[2])=h^{0}(\omy[4](2X))-h^{0}(\omy[4](X)) -
  h^{0}(\omy[3](X)) + h^{1}(\omy[3](X))-\dim(\im\beta+\im\delta).\]
Repeating the computations for a smooth complete intersection
$X_{\operatorname{smooth}}$ of the same type we get 
\[h^{1}(\Omega^{2}_{X_{\rm smooth}})=h^{0}(\omy[4](2X))-h^{0}(\omy[4](X)) -
  h^{0}(\omy[3](X)) + h^{1}(\omy[3](X))\]
so by previous Proposition
\[h^{1,2}(\tx)=h^{1,2}(X_{\rm smooth})-\mu+\delta.\]
As $\tx$ is the blow--up of $\mu$ lines in any small resolution $\hat X$
we get formula for $h^{1,2}(\hat X)$, formula for $h^{1,1}(\hat X)$
follows now from an easy Milnor number computation.
\end{proof}

\section{Examples}
Defect formula in main theorem can be easily implemented in a computer
algebra system, we use Magma code (\cite{Magma}).
\begin{exmp}
  Denote by $X(d_{1},\dots,d_{r};e_{1},\dots,e_{r+1})$ general
  complete intersection of hypersurfaces of degrees
  $d_{1},\dots,d_{r}$ in $\PP^{r+3}$ containing general complete
  intersection surface of degrees $e_{1},\dots,e_{r+1}$. 
  In \cite{CR2} these nodal threefolds were studied as candidates
  for non--factorial nodal complete intersections with minimal number
  of nodes (cf. \cite{Kl, Ko}). Using our main result we check that the
  defect equals 1 for the following cases with $r=2$ and
  $d_{1}+d_{2}=6$ (Calabi--Yau cases), 

\[
  \begin{array}{c@{\ }c@{\ }c@{\ }c@{\ }c|c|c|c}
    d_{1}&\ d_{2}&\ e_{1}&\ e_{2}&\ e_{3} &
    \mu&h^{1,1}&h^{1,2}\\ 
    \hline
    4&2&1&1&1&13&2&77\\
    4&2&2&1&1&18&2&72\\
    4&2&2&2&1&24&2&66\\
    4&2&2&2&2&32&2&58\\
    4&2&3&2&1&18&2&72\\
    4&2&3&2&2&24&2&66\\
    4&2&3&3&2&18&2&72\\    
  \end{array}
\]
\end{exmp}
\begin{exmp}We use our main result to verify computations of the Hodge
  numbers of some rigid Calabi--Yau complete intersections.

  Complete intersection of four quadrics in
  projective space $\PP^{7}$
  \[\begin{array}{c}
Y_{0}^{2}=X_{0}^{2}+X_{1}^{2}+X_{2}^{2}+X_{3}^{2}\\
Y_{1}^{2}=X_{0}^{2}-X_{1}^{2}+X_{2}^{2}-X_{3}^{2}\\
Y_{2}^{2}=X_{0}^{2}+X_{1}^{2}-X_{2}^{2}-X_{3}^{2}\\
Y_{3}^{2}=X_{0}^{2}-X_{1}^{2}-X_{2}^{2}+X_{3}^{2}\\
\end{array}
\]
studied by van Geemen and Nygaard in \cite{vGN}. Using counting points
in characteristic 17 they proved that small resolution of this
complete intersection is rigid, i.e. $h^{1,1}=32, h^{1,2}=0$. 
The Hodge numbers of a smooth complete intersection of four quadrics
equal 
\[h^{1,1}=1, h^{1,2}=65.\]
Using magma code we compute 
\[\dim_{\CC}I^{5}=144, \dim_{\CC}(I\cap\mathcal J_{\Sigma})^{5}=79,
  \mu=96, \delta=96-(144-79)=31\]
and finally for the Hodge numbers of the van Geemen Nygaard complete
intersection equals
\[h^{1,1}=1+31=32,\ \  h^{1,2}=65-96+31=0.\]
as computed in \cite{vGN}.

For the complete intersection of a quadric in quartic in $\PP^{5}$
given by \cite{WvG} 
  \begin{eqnarray*}
    x_{1}^{2}+x_{2}^{2}+x_{3}^{2}&=&x_{4}^{2}+x_{5}^{2}+x_{6}^{2}\\
    x_{1}^{4}+x_{2}^{4}+x_{3}^{4}&=&x_{4}^{4}+x_{5}^{4}+x_{6}^{4}
  \end{eqnarray*}
In this case
\[\dim_{\CC}I^{5}=200, \dim_{\CC}(I\cap\mathcal J_{\Sigma})^{5}=111,
  \mu=122, \delta=122-(200-111)=33\]
and 
\[h^{1,1}=1+33=34,\ \  h^{1,2}=89-122+33=0.\]

Desingularized self fiber product of the Beauville surface $\Gamma(3)$
(constructed by Schoen in \cite{schoen}) is birational to the complete intersection
\begin{eqnarray*}
  x_{1}^{3}+x_{2}^{3}+x_{3}^{3}&=&x_{4}^{3}+x_{5}^{3}+x_{6}^{3}\\
  x_{1}x_{2}x_{3}&=&x_{4}x_{5}x_{6}
\end{eqnarray*}
with 108 nodes. We get
\[\dim_{\CC}I^{5}=219, \dim_{\CC}(I\cap\mathcal J_{\Sigma})^{5}=146,
  \mu=108, \delta=108-(219-146)=35\]
and 
\[h^{1,1}=1+35=36,\ \ h^{1,2}=73-108+35=0.\]
\end{exmp}
We have also computed Hodge numbers of nodal complete intersections
studied in \cite[Ch.~5]{Meyer} confirming Meyer's results.


\begin{thebibliography}{99}
\bibitem{Magma}W.~Bosma, J.~Cannon, C.~Playoust, \emph{The Magma
    algebra system. I. The user language}, J. Symbolic Comput., 24
  (1997), 235--265. 
\bibitem{Chelt}  I.~Cheltsov, \emph{On factoriality of nodal
threefolds.}
 J. Alg. Geom. \textbf{14} (2005), 663--690.
\bibitem{Clemens} C.~H.~Clemens, \emph{Double solids.} Adv.\ in
  Math. \textbf{47} (1983), 107--230. 
\bibitem{Cynk2} S.~Cynk, \emph{Defect of a nodal hypersurface.}
Manuscripta Math. \textbf{104} (2001), 325--331.
\bibitem{CR} \emph{Defect via differential forms with logarithmic
    poles},  Math. Nachr. 284 (2011), no. 17--18, 2148--2158. 
\bibitem{CR2} S.~Cynk, S.~Rams, \emph{Non--factorial nodal complete
    intersection threefolds}. Commun. Contemp. Math. 15 (2013), no. 5,
  1250064. 
\bibitem{digennaro} V.~Di Gennaro, D.~Franco, \emph{Factoriality
and N\'eron-Severi groups.}
Commun. Contemp. Math.  \textbf{10} (2008),  745--764.
\bibitem{Dimca} A.~Dimca, \emph{Betti numbers of hyperplanes and defects of
linear systems}, Duke Math. Jour. \textbf{60} (1990),285--294.
\bibitem{Dimca2} A.~Dimca, \emph{Singularities and topology of hypersurfaces.} Springer 1992.
\bibitem{esnview}  H.~Esnault, E.~Viehweg, \emph{Lectures on vanishing
    theorems.} Birkh\"auser 1992.
\bibitem{vGN} B.~van Geemen, N.~Nygaard,  
\emph{On the geometry and arithmetic of some Siegel modular threefolds},
Journal of Number Theory {\bf 53}   (1995), 45--87.
\bibitem{vGW}B.~van Geemen, J.~Werner, \emph{Nodal quintics in
    P4. Arithmetic of complex manifolds} (Erlangen, 1988), 48--59,
  Lecture Notes in Math., 1399, Springer, Berlin, 1989. 
\bibitem{Hirz2}F.~Hirzebruch, \emph{Some examples of threefolds with
    trivial canonical bundle}. In: Hirzebruch, F. (ed.). Gesammelte
  Abhandlungen. Collected papers, vol. II, pp. 757--770. Berlin
  Heidelberg New York: Springer 1987. 
\bibitem{Hirz} F.~Hirzebruch, \emph{Topological methods in algebraic
  geometry}. Reprint of the 1978 edition. Classics in
  Mathematics. Springer-Verlag, Berlin, 1995.  
\bibitem{hulek}  K.~Hulek, R.~Kloosterman,  \emph{Calculating the
    Mordell-Weil rank of elliptic threefolds and the cohomology of
    singular hypersurfaces.} Preprint available at arXiv:
  math/0806.2025, 2008. 
\bibitem{Kl} R.~Kloosterman, \emph{Nodal complete intersection threefold with
defect}
\bibitem{Ko} D. Kosta. \emph{Factoriality of complete intersections in
    $\mathbb P^{5}$}. Tr. Mat. Inst. Steklova 264, 109--115, 2009.
\bibitem{Meyer} C.~Meyer, \emph{Modular Calabi-Yau Threefolds},
Fields Institute Monograph \textbf{22} (2005), AMS.
\bibitem{pet-st} C.~Peters, J.~Steenbrink, \emph{Infinitesimal
    variations of Hodge structure and the generic Torelli problem 
for projective hypersurfaces.}  Classification of
algebraic and analytic manifolds (Katata, 1982), 399--463, Progr. Math. 39, Birkh\"auser 1983.
\bibitem{pet-st-book} C.~Peters, J.~Steenbrink, \emph{Mixed Hodge Structures.} Springer 2008
\bibitem{cheltsov2} V.~V.~Przhiyalkovskii, I.~Cheltsov, K.~A.~Shramov,
 \emph{Hyperelliptic and trigonal Fano threefolds.} (Russian)
Izv. Ross. Akad. Nauk Ser. Mat.  \textbf{69}
(2005), 145--204.
\bibitem{rams-hab} S.~Rams,
\emph{Defect and Hodge numbers of hypersurfaces.}
 Adv. Geom. \textbf{8} (2008), 257-288.
\bibitem{reid}M.~Reid, Miles,
\emph{The moduli space of 3-folds with K=0 may nevertheless be irreducible}. 
Math. Ann. 278 (1987), no. 1-4, 329--334. 
\bibitem{schoen}C.~Schoen,
\emph{On fiber products of rational elliptic surfaces with section}. 
Math. Z. 197 (1988), no. 2, 177--199. 
\bibitem{vanstrquint}  D. van Straten,  \emph{A quintic hypersurface in $\PP_4$ with 130 nodes},
Topology~\textbf{32} (1993), 857-864.
\bibitem{Werner} J. Werner, \emph{Kleine Aufl\"osungen spezieller
  dreidimensionaler Variet\"aten}, Bonner Math. Scriften \textbf{186}
(1987). 
\bibitem{WvG}J.~Werner, B.~van Geemen, \emph{New examples of
    threefolds with c1=0}. Math. Z. 203 (1990), no. 2, 211--225.
\end{thebibliography}
\end{document}